\documentclass[preprint,11pt]{elsarticle}
\usepackage[T1]{fontenc}
\usepackage[utf8]{inputenc}
\usepackage[english]{babel}
\usepackage{amsmath,amsfonts,amsthm,amssymb}
\usepackage{vmargin}
\usepackage{enumerate}
\usepackage{url}
\usepackage{stmaryrd}
\usepackage{tikz,pgf,tkz-berge}

\newcommand\blank[1][.6em]{%
  \mbox{\kern.06em\vrule height.5ex}%
  \vbox{\hrule width#1}%
  \hbox{\vrule height.5ex}}
  
\begin{document}
\begin{frontmatter}

\title{List coloring digraphs\tnoteref{thanks}}
\tnotetext[thanks]{The results presented in this paper are part of the Master Degree report of the third author at LIP (ENS de Lyon).
The first author was supported by ERC Advanced Grant GRACOL, project no. 320812.
The second author was supported by an FQRNT postdoctoral research grant and CIMI research fellowship.}

\author[dtu]{Julien Bensmail}
\author[toulouse]{Ararat Harutyunyan}
\author[lip]{Ngoc Khang Le}

\address[dtu]{Department of Applied Mathematics and Computer Science \\ Technical University of Denmark \\ DK-2800 Lyngby, Denmark \\~}
\address[toulouse]{Institut de Math\'ematiques de Toulouse \\ Université Toulouse III \\ 31062 Toulouse Cedex 09, France\\~}
\address[lip]{Laboratoire d'Informatique du Parall\'elisme \\ \'Ecole Normale Sup\'erieure de Lyon \\ 69364 Lyon Cedex 07, France\\~}

\begin{abstract}
The dichromatic number $\vec{\chi}(D)$ of a digraph $D$ is the least number $k$ such that the 
vertex set of $D$ can be partitioned into $k$ parts each of which induces an acyclic subdigraph. 
Introduced by Neumann-Lara in 1982, 
this digraph invariant shares many properties with the usual chromatic number of graphs and can be seen as the natural analog of the graph chromatic number. In this paper, we study the list dichromatic number of digraphs, giving evidence that this notion generalizes the list chromatic number of graphs. 
We first prove that the list dichromatic number and the dichromatic number behave the same in many contexts,
such as in small digraphs (by proving a directed version of Ohba's Conjecture), tournaments, and random digraphs.
We then consider bipartite digraphs, and show that their list dichromatic
number can be as large as $\Omega(\log_2 n)$.
We finally give a Brooks-type upper bound on the list dichromatic number of digon-free digraphs.
\end{abstract}

\end{frontmatter}

\newtheorem{theorem}{Theorem}
\newtheorem{lemma}[theorem]{Lemma}
\newtheorem{conjecture}[theorem]{Conjecture}
\newtheorem{observation}[theorem]{Observation}
\newtheorem{claim}[theorem]{Claim}
\newtheorem{corollary}[theorem]{Corollary}
\newtheorem{proposition}[theorem]{Proposition}
\newtheorem{question}[theorem]{Question}
\numberwithin{theorem}{section}

\newcommand{\tDelta}{\tilde{\Delta}}
\newcommand{\PP}{\mathbb{P}}
\newcommand{\EE}{\mathbb{E}}


\section{Introduction} \label{section:introduction}

A \textit{$k$-list-assignment} to a graph (or digraph) $G$ is an assignment of a set of positive integers $L(v)$ to each vertex $v$ such that $|L(v)| \geq k$ for each vertex $v$.
A \textit{proper} \textit{$k$-}\textit{vertex-coloring} of an undirected (simple) graph $G$ is a 
partition $V_1,...,V_k$ of $V(G)$ into $k$ independent sets. The \textit{chromatic number} of $G$, denoted $\chi(G)$, is the least~$k$ such that $G$ is properly 
$k$-vertex-colorable. The chromatic number has often been considered in a more general context. 
Namely, we say that $G$ is \textit{$k$-list-colorable} if for any $k$-list-assignment $L$ to $G$, 
there is a proper vertex-coloring such that each vertex is assigned a color from its list. 
The \textit{list chromatic number} $\chi_\ell(G)$ of $G$ is then the least $k$ such that $G$ is properly $k$-list-colorable. 
Many papers have studied how the chromatic number and list chromatic number parameters behave in general 
(see notably~\cite{ERT79} by Erd\H{o}s, Rubin and Taylor, who have introduced the subject). 

Our purpose in this paper is to study list coloring of directed graphs.
The digraphs considered in this paper will not have loops or parallel arcs -
we do, however, allow directed cycles of length~$2$ (or \textit{digons}). 
A subset $S$ of vertices of a digraph $D$ is called \textit{acyclic} if the induced
subdigraph on $S$ contains no directed cycle. The \textit{dichromatic number} $\vec{\chi}(D)$ of $D$ is the smallest integer $k$ such that
$V(D)$ can be partitioned into $k$ sets $V_1, ..., V_k$ where each $V_i$ is acyclic. 
Note that, equivalently, the dichromatic number is the smallest integer $k$ such that the vertices of $D$ can be colored with
$k$ colors so that there is no monochromatic directed cycle.
It is easy to see that for any undirected graph $G$ and 
its bidirected digraph $D$ obtained from $G$ by replacing each edge by two oppositely oriented arcs, we have $\chi(G) = \vec{\chi}(D)$.
The dichromatic number was first introduced by Neumann-Lara~\cite{Neu82}
in 1982. It was independently rediscovered by Mohar \cite{M2003} who showed
that this parameter allows natural generalization of circular colorings of graphs to digraphs. Further studies by Bokal \textit{et al.}~\cite{BFJKM04} 
and Harutyunyan and Mohar (see notably~\cite{HM11,HM11b,HM12,HM15,M10}) demonstrated that this digraph invariant generalizes
many results on the graph chromatic number. Some of the results obtained
in the above mentioned papers include extensions of 
Brooks' Theorem, Wilf's Eigenvalue Theorem, Gallai's Theorem,
as well as examples of classical theorems from extremal graph theory.  

Here, we study the \emph{list dichromatic number} of digraphs. Given a $k$-list-assignment $L$ to a digraph $D$,
we say that $D$ is \emph{$L$-colorable} if there is a coloring of the vertices such that each vertex $v$ is colored with a color from $L(v)$ and
$D$ does not contain a monochromatic directed cycle. The \emph{list dichromatic number} $\vec{\chi_{\ell}}(D)$ of $D$ is the smallest
integer $k$ such that $D$ is $L$-colorable for any $k$-list-assignment $L$. We note that the definition of
list coloring of digraphs is not quite new as it first appeared in \cite{HM11b} where
the authors derived an analog of Gallai's Theorem for digraphs, as well as
in \cite{HM15}.
Our goal in this paper is to initiate the study of the list dichromatic number. In particular, we point out that the relationship 
between the dichromatic number and list dichromatic number is surprisingly
similar to the relationship between 
the graph chromatic number and the list chromatic number. This is mainly done by extending known results involving the chromatic number and 
list chromatic number to directed graphs.

\medskip

This paper is organized as follows. 
We start, in Section~\ref{section:ohba}, by showing that the equality $\vec{\chi}=\vec{\chi_\ell}$ holds for small digraphs,
by proving a directed version of Ohba's Conjecture.
We then consider bipartite digraphs in Section~\ref{section:bipartite},
and in particular show that while $\vec{\chi} \leq 2$ holds for these digraphs, 
their value of $\vec{\chi_\ell}$ can be as large as $\Omega(\log_2 n)$, obtaining a natural digraph analogue of a known result in \cite{ERT79}.
In Sections~\ref{section:tournaments} and~\ref{section:random}, we consider tournaments
and random digraphs, respectively, and show that their parameters $\vec{\chi}$ and $\vec{\chi_\ell}$ behave in a similar way.
We then provide, in Section~\ref{section:upper}, a Brooks-like upper bound on $\vec{\chi_\ell}$ for digon-free
digraphs, extending an analogous bound on the dichromatic number in
\cite{HM11}.
We finally end this paper by addressing an intriguing question in concluding Section~\ref{section:ccl}.
Most of the tools and arguments employed in our proofs are probabilistic; for more details,
we refer the reader to the book of Molloy and Reed~\cite{MR02}.

\section{Ohba's Conjecture for digraphs} \label{section:ohba}

In~\cite{O02}, Ohba conjectured that the equality $\chi=\chi_\ell$ should hold for every undirected graph with relatively small order
(compared to their chromatic number). Ohba's Conjecture was recently proved in \cite{NRW14} by Noel, Reed and Wu. 

\begin{theorem}[Ohba's Conjecture] \label{C:2:2}
For every graph $G$ with $|V(G)|\leq 2\chi(G)+1$, we have $\chi(G)=\chi_\ell(G)$.
\end{theorem}

\noindent 

One could naturally wonder whether the directed analog of Ohba's Conjecture holds in our context.
We prove this to be true in the following result.

\begin{theorem} \label{theorem:Ohba-digraph}
For every digraph $D$ with $|V(D)|\leq 2\vec{\chi}(D)+1$, we have $\vec{\chi}(D)=\vec{\chi_\ell}(D)$.
\end{theorem} 

\begin{proof}
Set $\vec{\chi}(D)=k$, and consider a partition $V_1, ..., V_k$ of $V(D)$ such that each $V_i$ is acyclic. Now consider
the following undirected graph $G_D$ obtained from $D$:

\begin{itemize}
	\item $V(G_D) = V(D)$,
	\item $uv \in E(G_D)$ iff $u \in V_i$ and $v \in V_j$ with $i \neq j$.
\end{itemize}

\noindent Note that $G_D$ is a complete $k$-partite graph with vertex partition $V_1, ..., V_k$. Therefore, $\chi(G_D)=k$. 
Furthermore, it is easy to see that $G_D$ satisfies the condition stated in Theorem \ref{C:2:2}, since:
\begin{equation*}
|V(G_D)| = |V(D)| \leq 2\vec{\chi}(D)+1 = 2k+1 = 2\chi(G_D)+1.
\end{equation*} 
\noindent Hence, $\chi(G_D)=\chi_\ell(G_D)=k$ since Theorem~\ref{C:2:2} is verified~\cite{NRW14}.

Now consider any $k$-list-assignment $L$ to $D$. Since $\chi_\ell(G_D)=k$, we can find an $L$-coloring $\phi$
of $G_D$, \textit{i.e.} an assignment of colors to the vertices from their respective list so that there is no monochromatic
edge connecting two parts $V_i$ and $V_j$ with $i \neq j$. It should be now clear that $\phi$ is also an $L$-coloring
of $D$:

\begin{itemize}
	\item since, for every edge $uv$ with $u \in V_i$, $v \in V_j$ and $i \neq j$, we have $\phi(u) \neq \phi(v)$,
	there cannot be any monochromatic cycle including vertices from different $V_i$'s;
	
	\item since each $V_i$ is acyclic, there cannot be any monochromatic cycle including vertices from a single $V_i$.
\end{itemize}

\noindent Thus, $\vec{\chi_\ell}(D)=k=\vec{\chi}(D)$.  
\end{proof}


\section{Bipartite digraphs} \label{section:bipartite}

One of the classical results in the theory of list coloring of graphs is the 
result due to Erd\H{o}s, Rubin and Taylor \cite{ERT79} 
which asserts that there exist
bipartite graphs $G$ with $\chi_\ell(G) = \Omega(\log_2 n)$.  
Indeed, they proved the following.

\begin{theorem}[Erd\H{o}s, Rubin and Taylor]
For every complete bipartite graph $K_{n,n}$, we have $\chi_\ell(K_{n,n}) = (1+o(1)) \log_2 n$.
\end{theorem}
  
Since $\chi(G) \leq 2$ holds for every bipartite graph $G$, clearly $\vec{\chi}(D) \leq 2$ holds for every bipartite digraph $D$ (by a \emph{bipartite digraph} we mean a digraph that can be obtained by orienting the edges of a bipartite graph). A natural question is then how large
can $\vec{\chi_\ell}(D)$ be for bipartite digraphs. Interestingly, we find that $\vec{\chi_\ell}(D)$ can also be as large as $\Omega(\log_2 n)$.

%

%


\subsection{Upper bounds}

We start by exhibiting the following upper bound on $\vec{\chi_\ell}$ for bipartite digraphs.

\begin{theorem} \label{theorem:upper-bound-bipartite}
For every bipartite digraph $D$ with bipartition $(V_1,V_2)$ satisfying $|V_1|=|V_2|=n$, $$\vec{\chi_\ell}(D)\leq \lfloor \log_2 n\rfloor + 2.$$
\end{theorem}

\begin{proof}
We use the probabilistic method. Assume $L$ is a $(\lfloor \log_2 n\rfloor + 2)$-list-assignment to $D$, and let $C$ 
denote the set of all colors (\textit{i.e.} $C=\cup_{v \in V(D)} L(v)$). We partition $C$ into two parts $C_1$ and $C_2$ that will be used to color the vertices in $V_1$ and $V_2$, respectively. To obtain $(C_1,C_2)$, just consider every color $c \in C$, and put it randomly in $C_1$ or $C_2$ with equal probability $1/2$.

For every vertex $v \in V_i$ of $D$, let $E_v$ be the event that $L(v)\cap C_i=\emptyset$. Clearly if none of the $E_v$'s occurs, then we can freely assign to each vertex $v$ a color from $L(v) \cap C_i$ since no monochromatic directed cycle can appear by the partition of $C$. We would hence like none of the $E_v$'s to happen. Since a particular event $E_v$ with, say, $v \in V_1$ 
can only occur if all colors of $L(v)$ were put into $C_2$, we deduce: $$\PP\left[\cup_{v\in V(D)}E_v\right] \leq 2n \cdot \PP[E_v] \leq \frac{2n}{2^{\lfloor \log_2 n\rfloor + 2}} < 1.$$
Hence, with positive probability none of the $E_v$'s happens, implying the claim.
\end{proof}
\medskip

Theorem \ref{theorem:upper-bound-bipartite} can be extended to non-bipartite
digraphs. Indeed, by a virtually identical argument one can prove the following.

\begin{theorem} \label{theorem:upper-bound-chi}
For every digraph $D$, we have $\vec{\chi_\ell}(D) \leq \vec{\chi}(D) \ln n$.
\end{theorem}

Note that Theorem~\ref{theorem:upper-bound-chi} is already known for the graph chromatic number,
\textit{i.e.} it is known that $\chi_\ell(G) \leq \chi(G) \ln n$ holds for every graph $G$.


\subsection{Lower bound}


In this section, we show that there exist bipartite digraphs $D$
with $\vec{\chi_\ell}(D) = \Omega (\log_2 n)$. We say that a random graph (or digraph) $G$ of order $n$ \textit{asymptotically almost surely (\textit{a.a.s.})} has a property $P$ if the probability that $G$ has $P$ tends to~$1$ as~$n$ approaches infinity.

\begin{theorem} \label{theorem:lower-random-bipartite}
Let $D$ be the random complete bipartite digraph $D$ with bipartition $(V_1,V_2)$ satisfying $|V_1|=|V_2|=n$ obtained from $K_{n,n}$ by orienting each edge
randomly, that is in either direction with probability $1/2$, and independently. Then \textit{a.a.s.} $$\vec{\chi_\ell}(D) = \Omega(\log_2 n).$$
\end{theorem}

We first prove the following useful lemma. 

\begin{lemma} \label{L:1:1}
Let $D$ be a random complete bipartite digraph with bipartition $(V_1,V_2)$ satisfying $|V_1|=|V_2|=n$. 
Then for every $V_1'\subset V_1$ and $V_2'\subset V_2$ verifying $|V_1'|,|V_2'|\geq 3\log_2 n$, \textit{a.a.s.} $D[V_1' \cup V_2']$ has a directed cycle.
\end{lemma} 

\begin{proof}
Let us recall that, by a result of Manber and Tompa~\cite{MT84}, 
the number of acyclic orientations of a given graph $G$ is at most
$$\prod_{v \in V(G)} (d(v) +1).$$
Now consider any two subsets $V_1'\subset V_1$ and $V_2'\subset V_2$ of size $3 \log_2 n$.
Since the number of possible $V_1'$ (and similarly, $V_2'$) is at most ${n \choose 3 \log_2 n}$,
by the remark above and the Union Bound, we get:
\begin{align*}
	\PP\left[\exists V_1', V_2' : D[V_1' \cup V_2'] {\rm ~is~acyclic}\right] & \leq {n \choose 3 \log_2 n}^2 \frac{(3 \log_2 n+1)^{6 \log_2 n}}{2^{9(\log_2 n)^2}} \\
	& \leq n^{6 \log_2 n} \frac{(3 \log_2 n +1)^{6 \log_2 n}}{2^{9(\log_2 n)^2}} \\
	& \leq 2^{6 (\log_2 n)^2} \frac{2^{6 \log_2 n \log_2 \log_2 n(1+o(1))}}{2^{9(\log_2 n)^2}} = o(1).
\end{align*}
Therefore, $D[V_1' \cup V_2']$ \textit{a.a.s.} has a directed cycle.
\end{proof}

\begin{proof}[Proof of Theorem~\ref{theorem:lower-random-bipartite}]
Let $D$ be the random complete bipartite digraph with bipartition $(V_1, V_2)$ satisfying $|V_1|=|V_2|=n=3k{2k-1\choose k}\log_2 n$. Now let
$L$ be a $k$-list-assignment to $D$ from a pool of $2k-1$ colors such that every possible list of $k$ colors is equally distributed among
all vertices, that is $n/{2k-1\choose k} = 3k\log_2 n$ vertices in $V_1$ (and similarly in $V_2$) have the exact same list.

Assume $D$ is $L$-colorable, and let $\phi$ be any $L$-coloring of $D$. By a \textit{major color} in $V_1$ (and similarly, in $V_2$) by $\phi$, we refer to
a color assigned to at least $3\log_2 n$ vertices of $V_1$ (resp. $V_2$). Clearly, there has to be at least $k$ major colors appearing in $V_1$, and similarly
in $V_2$. Otherwise, assuming $\cup_{v \in V(D)} L_v = \{1, 2, ..., 2k-1\}$ and only colors $1, 2, ..., k-1$ are major in, say, $V_1$, there would remain, 
by construction, $3k\log_2 n$ vertices of $V_1$ with list $\{k, k+1, ..., 2k-1\}$ -- so one of these colors necessarily has to be major,
a contradiction. So there exists a common major color in $V_1$ and $V_2$ since $L$ takes value among a pool of $2k-1$ colors only.
Denoting $V_1'$ and $V_2'$ the subsets of vertices from $V_1$ and $V_2$, respectively, assigned that major color by $\phi$, by Lemma~\ref{L:1:1}
we get that \textit{a.a.s.} $D[V_1' \cup V_2']$ has a monochromatic directed cycle. So $\phi$ is not an $L$-coloring, and we get that $\vec{\chi_\ell}(D) > k$, as claimed. An easy computation shows that we furthermore have $k=\Theta(\log_2 n)$.
\end{proof}
\medskip

\textbf{Remark:} Theorem \ref{theorem:lower-random-bipartite} can be strengthened. 
In fact, by adapting an argument of Alon and Krivilevich \cite{AK98} on the list chromatic number of random bipartite graphs and using a sligthly strengthened
version of Lemma \ref{L:1:1}, one can obtain the optimal lower bound of $
\vec{\chi_\ell}(D) \geq \log_2 n (1+o(1))$.

\section{Tournaments} \label{section:tournaments}

We here prove that the list dichromatic number of tournaments asymptotically behaves like their dichromatic number. 
We first prove this roughly for all tournaments in Section~\ref{subsection:gen-tournaments}, 
before showing, in Section~\ref{subsection:rand-tournaments}, something more accurate in the context of random tournaments.


\subsection{General tournaments} \label{subsection:gen-tournaments}

We start off by proving an upper bound on $\vec{\chi_\ell}$ for every tournament.

\begin{theorem} \label{theorem:all-tournaments-upbound}
For every tournament $T$ with order $n$, $$\vec{\chi_\ell}(T) \leq \frac{n}{\log_2 n}(1+o(1)).$$
\end{theorem}  

\noindent Our proof of Theorem~\ref{theorem:all-tournaments-upbound} relies on the following lemma, due to Erd\H{o}s and Moser \cite{EM64}.

\begin{lemma}[Erd\H{o}s and Moser] \label{lemma:tournament-acyclic-size}
Every tournament $T$ with order $n$ has an acyclic set of size at least $\log_2 n + 1$.
\end{lemma}
%

\begin{proof}[Proof of Theorem \ref{theorem:all-tournaments-upbound}]
Assume $T$ is given an $\frac{n}{\log_2 n}(1+\frac{1}{\log_2 \log_2 \log_2 n})$-list-assignment, and apply the following coloring process:
\begin{enumerate}
	\item If there exists a set $S$ of at least $\left \lfloor \frac{n}{(\log_2 n)^2} \right \rfloor$ uncolored vertices all of which contain a common color~$x$ in their list:
	\begin{enumerate}
		\item Pick an acyclic set $A \subseteq S$ of size at least $\log_2 n - 2 \log_2 \log_2 n$ -- such a set exists by Lemma \ref{lemma:tournament-acyclic-size}.
		\item Assign color $x$ to every vertex in $A$ and remove $x$ from all lists.
	\end{enumerate}
	\item Repeat Step~1. as long as possible.
\end{enumerate}

To conclude the proof, it remains to show that we can properly color the remaining vertices $R$ that are still uncolored at the end of the process. Let $C$ be the set of remaining colors in the lists of the vertices in $R$. Now construct the bipartite graph $G$ with bipartition $(R,C)$ in which there is an edge between a vertex $r \in R$ and a vertex $c \in C$ if and only if $c$ belongs to the color list of $r$ after the above procedure. By the halting condition, for every $c \in C$ we clearly have: 
\begin{equation} \label{E:1}
d_G(c)\leq \left \lfloor \frac{n}{(\log_2 n)^2} \right \rfloor.
\end{equation}    
Furthermore, for every $r \in R$, we have:
\begin{align} 
d_G(r) &\geq \frac{n}{\log_2 n}\left(1+ \frac{1}{\log_2 \log_2 \log_2 n}\right) - \frac{n}{\log_2 n - 2 \log_2 \log_2 n} \nonumber \\
&\geq \left \lfloor \frac{n}{(\log_2 n)^2} \right \rfloor.\label{E:2}
\end{align}

Now consider any subset $B \subseteq R$, and let us denote by $N(B)$ the set of vertices of $C$ neighboring some vertices in $R$, and by $e(B,N(B))$ the set of edges of $G$ joining vertices of $B$ and $N(B)$. From Inequalities~(\ref{E:1}) and~(\ref{E:2}), we deduce $$|B| \cdot \left \lfloor \frac{n}{(\log_2 n)^2} \right \rfloor \leq e(B,N(B))\leq |N(B)| \cdot \left \lfloor \frac{n}{(\log_2 n)^2} \right \rfloor,$$ and hence $|B|\leq |N(B)|$. Therefore, by Hall's Theorem, $G$ admits a matching hitting all vertices of $R$. In other words, there is a way to color the vertices of $R$ with colors from their respective lists in such a way that each of the remaining colors is used at most once. Such a coloring, together with the partial coloring resulting from the coloring procedure above, yields an acyclic coloring of $T$.
\end{proof}


\subsection{Random tournaments} \label{subsection:rand-tournaments}

By a \textit{random tournament} of order $n$, we mean a tournament obtained from $K_n$ by orienting each edge $uv$ 
either from $u$ to $v$ or from $v$ to $u$ with equal probability $1/2$. 
We show that the bound in Theorem \ref{theorem:all-tournaments-upbound} is tight up to some constant in the sense of the following result. 
We remark that Theorem \ref{theorem:random-tournaments-lowbound} is not new -- 
for example, it appears in \cite{H11}, where the proof is based on a slightly weaker version in \cite{EGK91}.  

\begin{theorem} \label{theorem:random-tournaments-lowbound}
Let $T$ be the random tournament of order $n$. Then \textit{a.a.s.}
\begin{equation*}
\vec{\chi}(T)\geq \frac{n}{2\log_2 n + 2}.
\end{equation*}
\end{theorem}

\begin{proof}
Let $A$ be any fixed subset of $2\log_2 n + 2$ vertices of $T$. Note that if the subdigraph $T[A]$ induced by $A$ is acyclic then there is 
an ordering of the vertices in $A$ such that all arcs of $T[A]$ go forward with respect to that ordering. Thus, assuming $\vec{\alpha}(T)$ is
the size of a largest acyclic set in $T$, we have:

\begin{eqnarray*}
\PP[\vec{\alpha}(T)\geq 2\log_2 n + 2] &\leq& {n \choose 2\log_2 n + 2} \PP[A \textrm{ is acyclic}] \\
&\leq&  {n \choose 2\log_2 n + 2} (2\log_2 n + 2)! \left(\frac{1}{2}\right)^{2\log_2 n + 2 \choose 2}\\
&\leq& n^{2\log_2 n + 2}\cdot\frac{1}{n^{2\log_2 n + 1}}\cdot\frac{1}{n^2} = o(1).
\end{eqnarray*}
Therefore, \textit{a.a.s.} $\vec{\alpha}(T)\leq 2\log_2 n + 2$. Since $\vec{\chi}(D) \geq \frac{|V(D)|}{\vec{\alpha}(D)}$ for any digraph $D$, the result follows.
\end{proof}
\medskip

Note that Theorem~\ref{theorem:random-tournaments-lowbound} provides a lower bound on the asymptotic value of $\vec{\chi_\ell}(T)$ for every random tournament $T$, 
since the list dichromatic number of every digraph is at least its dichromatic number. 
We now prove an upper bound on the asymptotic value of $\vec{\chi_\ell}(T)$.

Here and further, we will need the so-called Extended Janson Inequality (see \textit{e.g.}~\cite{AS04}), which we recall formally now.
Let $\Omega$ be a finite universal set and let $R \subseteq \Omega$ be obtained as follows. For every $r \in \Omega$, we independently add $r$ to $R$
with probability $p_r$. Given a finite set $I$ of indexes, let $A_i$ be subsets of $\Omega$ for $i \in I$, and denote by $B_i$ the event that $A_i \subseteq R$.
Let $X_i$ be the indicator random variable for $B_i$, \textit{i.e.} $X_i=1$ if $B_i$ occurs, and $X_i=0$ otherwise. Set $X = \sum_{i \in I} X_i$.
For any two $i,j \in I$, we write $i \sim j$ if $i \neq j$ and $A_i \cap A_j \neq \emptyset$. We define $$\Delta=\sum_{i \sim j} \PP[B_i \cap B_j],$$ 
and finally set $$\mu = \EE[X] = \sum_{i \in I} \PP[B_i].$$ We are now ready to state the two Janson Inequalities.

\begin{theorem}[Janson Inequality] \label{theorem:janson}
Let $B_i$ for $i \in I$, and $\Delta$ and $\mu$ be as above. Then $$\PP[X=0]=\PP[\cap_{i \in I} \bar{B_i}] \leq e^{-\mu+\frac{\Delta}{2}}.$$
\end{theorem}

\begin{theorem}[Extended Janson Inequality]
Under the assumption of Theorem~\ref{theorem:janson} and the further assumption that $\Delta \geq \mu$, 
$$\PP[X=0] = \PP[\cap_{i \in I} \bar{B_i}] \leq e^{-\frac{\mu^2}{2\Delta}}.$$
\end{theorem}

\begin{theorem} \label{theorem:random-tournaments-upbound}
Let $T$ be the random tournament of order $n$. Then \textit{a.a.s.}
\begin{equation*}
\vec{\chi_\ell}(T)\leq \frac{n}{2\log_2 n}(1+o(1)).
\end{equation*}
\end{theorem}

\begin{proof}
The proof is an adaptation of an argument used in~\cite{AS04} for proving a bound on the chromatic number of random graphs, which uses the Extended Janson Inequality. 

In the current proof, we will write $a\thicksim b$ if $a/b$ tends to 1 as $n$ tends to infinity. Let $k\thicksim 2\log_2 n$, the exact value of $k$ being indicated later. For each $k$-set $S$ of vertices (\textit{i.e.} $S$ has size $k$), let $A_S$ be the event ``$S$ is an acyclic set'' and $X_S$ be the corresponding indicator random variable. For any additional $k$-set $S'$, we will write $S\thicksim S'$ if $|S \cap S'|\geq 2$. 

Set
\begin{equation*}
X = \sum_{|S|=k}X_S.
\end{equation*}
Note that $S$ is acyclic if and only if there is an ordering of the vertices in $S$ such that all arcs in $S$ go forward. Then $\PP[A_S]=\frac{k!}{2^{k \choose 2}}$. Hence
\begin{equation*}
\EE[X] = f(k) = {n \choose k}2^{-{k \choose 2}}k!.
\end{equation*}
Then, $\vec{\alpha}(D)\geq k$ if and only if $X>0$. We have
\begin{equation*}
\Delta = \sum_{S\thicksim S'}\PP[A_S\cap A_{S'}] = \sum_{S}\PP[A_S]\sum_{S'\thicksim S}\PP[A_{S'}|A_S].
\end{equation*}
Note that the inner summation is independent of $S$. Now we set
\begin{equation*}
\Delta^* = \sum_{S'\thicksim S}\PP[A_{S'}|A_S]
\end{equation*}
where $S$ is any fixed $k$-set. Then
\begin{equation*}
\Delta = \sum_{S}\PP[A_S]\Delta^* = \Delta^*\sum_{S}\PP[A_S] = \Delta^* \EE[X].
\end{equation*}
If we assume that $S$ is acyclic and $|S \cap S'| = i$, then the number of ways to order the vertices of $S'$ in an acyclic way is ${k \choose i}(k-i)! = k!/i!$. 
From the definition of $\Delta^*$, we have
\begin{equation*}
\Delta^* = \sum_{i=2}^{k-1}{k \choose i}{n-k \choose k-i}2^{{i \choose 2}-{k \choose 2}}\frac{k!}{i!}
\end{equation*}
and so
\begin{equation*}
\frac{\Delta^*}{\EE[X]} = \sum_{i=2}^{k-1}g(i)
\end{equation*}
where we set
\begin{equation*}
g(i) = \frac{{k \choose i}{n-k \choose k-i}2^{i \choose 2}}{{n \choose k}i!}.
\end{equation*}

Let $k_0=k_0(n)$ be the value for which:
\begin{equation*}
f(k_0-1)>1>f(k_0).
\end{equation*}
Note that $k_0\thicksim 2\log_2 n$. We have:
\begin{equation*}
\frac{f(k+1)}{f(k)} \thicksim n2^{-k} = n^{-1+o(1)}.
\end{equation*}
Set $k = k(n) = k_0(n)-4$. Then $f(k)>n^{3+o(1)}$. By computation, $g(2)\thicksim k^4/n^2$ and $g(k-1)\thicksim 2k^2n2^{-k}/\EE[X]$ are the dominating terms. 
In our instance, we have $\EE[X] = f(k) > n^{3+o(1)}$ and $2^{-k}=n^{-2+o(1)}$. So $g(2)$ dominates and
\begin{equation*}
\frac{\Delta}{\EE[X]^2} = \frac{\Delta^*}{\EE[X]} \thicksim \frac{k^4}{n^2}. 
\end{equation*}
Now, since $k=\Theta(\ln n)$, by the Extended Janson Inequality we deduce
\begin{equation*}
\PP[\vec{\alpha}(T)<k] = \PP[X=0] < e^{-\EE[X]^2/2\Delta} = e^{-\Theta(n^2/(\ln n)^4)} = e^{-n^{2+o(1)}}.
\end{equation*}

\medskip

In order to finish the proof, we need the following key lemma:

\begin{lemma} \label{L:2:2}
Let $m=\left \lfloor \frac{n}{(\log_2 n)^2} \right \rfloor$. \textit{A.a.s.} every $m$-set of vertices of $T$ contains an acyclic $k$-set, where $k = k(m) = k_0(m)-4$ as above.
\end{lemma}

\begin{proof}
For any $m$-set $S$, the induced subdigraph $T[S]$ has the distribution of the random tournament of order $m$. By the above argument concerning the Extended Janson Inequality, we have
\begin{equation*}
\PP[\vec{\alpha}(T[S])<k] <  e^{-m^{2+o(1)}}.
\end{equation*}
There are ${n\choose m}< 2^n = 2^{m^{1+o(1)}}$ such sets $S$. Hence
\begin{equation*}
\PP[\vec{\alpha}(T[S])<k \textrm{ for some set } S] <  2^{m^{1+o(1)}}e^{-m^{2+o(1)}} = o(1),
\end{equation*}
concluding the proof.
\end{proof}
\medskip

We now use the same argument as in the proof of Theorem \ref{theorem:all-tournaments-upbound}. Suppose $T$ is given an $\frac{n}{2\log_2 n}(1+o(1))$-list-assignment. Let $m= \lfloor \frac{n}{(\log_2 n)^2}  \rfloor$. If there exists a color appearing in the list of at least $m$ vertices, then we can find an acyclic set $S$ among these vertices and assign that color to every vertex of $S$. We repeat this process as long as possible until no such color exists. So the process ends with every remaining color appearing in the list of at most $m$ vertices. Since
\begin{equation*}
\frac{n}{2\log_2 n}(1+o(1)) - \frac{n}{2\log_2 n(1+o(1))} \geq \left \lfloor \frac{n}{(\log_2 n)^2} \right \rfloor ,
\end{equation*}
then we can assign to each remaining uncolored vertex a color from its list so that we obtain an acyclic vertex-coloring.
This completes the proof.
\end{proof}
\medskip

Theorems~\ref{theorem:random-tournaments-lowbound} and~\ref{theorem:random-tournaments-upbound} yield the following direct corollary.

\begin{theorem} \label{theorem:random-tournaments-asym}
Let $T$ be the random tournament of order $n$. Then \textit{a.a.s.}
\begin{equation*}
\vec{\chi}(T) \thicksim \vec{\chi_\ell}(T) \thicksim \frac{n}{2\log_2 n}.
\end{equation*}
\end{theorem}

\section{Random digraphs} \label{section:random}

We here focus on random digraphs, mainly confirming the equality $\vec{\chi} = \vec{\chi_\ell}$ for
the following model $D(n,p)$. An $n$-vertex digraph $D$ of $D(n,p)$ is constructed by connecting $n$ vertices
randomly as follows. For every two vertices $u$ and $v$ of $D$, the probability that $u$ and $v$ are connected
by an arc is $2p$. In case $u$ and $v$ are chosen to be joined by an arc, the direction of that arc is chosen 
with equal probability (\textit{i.e.} $1/2$ for $(u,v)$ and $1/2$ for $(v,u)$). All directions are chosen independently
for different pairs of vertices to be joined. Note that $D(n,1/2)$ generates random tournaments.

Before stating our main result, let us recall two results. The first result, due to Spencer and Subramanian~\cite{SS08},
gives a bound on the maximum size $\vec{\alpha}(D)$ of an induced acyclic subgraph in some given digraph $D$.

\begin{theorem}[Spencer and Subramanian] \label{T:1:2}
Let $D \in D(n,p)$ and $w=np$. There is a sufficiently large constant $C$ such that if $p$ satisfies $w\geq C$, then \textit{a.a.s.}
\begin{equation}
\vec{\alpha}(D) = \left(\frac{2\ln w}{\ln q}\right)(1\pm o(1)),
\end{equation}
where $q = (1-p)^{-1}$.
\end{theorem} 

The second result, due to Alon, Krivelevich and Sudakov~\cite{AKS99}, yields a bound on the list chromatic number of the random graph $G(n,p)$. 

\begin{theorem}[Alon, Krivelevich and Sudakov] \label{T:1:3}
If $p=p(n)$ satisfies $2<np\leq n/2$, then \textit{a.a.s.}
\begin{equation}
\chi_\ell(G(n,p))=\Theta\left(\frac{np}{\ln(np)}\right).
\end{equation}
\end{theorem}

We are now ready to prove our main result on random digraphs.

\begin{theorem}
Let $D \in D(n,p)$. There is a sufficiently large constant $C$ such that if $p=p(n)$ satisfies $C\leq np\leq n/4$, then \textit{a.a.s.}
\begin{equation}
\vec{\chi}(D) = \Theta \left(\frac{n\ln q}{\ln(np)}\right) \textrm{ and ~} \vec{\chi_\ell}(D) = \Theta \left(\frac{n\ln q}{\ln(np)}\right),
\end{equation}
where $q = (1-p)^{-1}$.
\end{theorem}

\begin{proof}
To begin with, from Theorem~\ref{T:1:2} we know that
\begin{equation} \label{E:1:1}
\vec{\chi}(D) \geq \frac{n}{\vec{\alpha}(D)} = \Omega \left(\frac{n\ln q}{\ln(np)}\right).
\end{equation}

Now, we know that the underlying undirected graph of $D(n,p)$ behaves somehow similarly to the graph $G(n,2p)$. 
Recall furthermore that the list dichromatic number of a digraph is not larger than the list chromatic number
of its underlying graph.
Thus, by Theorem~\ref{T:1:3}, and because we have $p \leq \ln q$ for all $0\leq p\leq 1/4$, we \textit{a.a.s.} get:

\begin{equation} \label{E:1:2}
\vec{\chi_\ell}(D) = O\left(\frac{np}{\ln(np)}\right) = O\left(\frac{n\ln q}{\ln(np)}\right).
\end{equation} 

The claim then follows from Equations~(\ref{E:1:1}) and~(\ref{E:1:2}), and the trivial fact that $\vec{\chi}(D) \leq \vec{\chi_\ell}(D)$.
\end{proof}

\section{A general upper bound on $\vec{\chi_\ell}$} \label{section:upper}

One common direction of research related to the chromatic number of graphs is to study how the chromatic number
of a graph $G$ behaves provided $G$ fulfils some particular properties. For example, while the chromatic number of a general
graph $G$ can be as large as $\Delta(G)+1$ (according to Brook's Theorem), it is known that $\chi(G)$ drops to $O(\frac{\Delta(G)}{\log_2 \Delta(G)})$ 
as soon as $G$ is triangle-tree (as proved by Johansson~\cite{Joh96}).
It is hence legitimate to wonder whether such a phenomenon can also be observed concerning the dichromatic number
of digraphs having no digons. Note that a complete bidirected
clique on $\Delta + 1$ vertices is $\Delta$-regular with
dichromatic number equal to $\Delta + 1$. It is also easy to see that
$\Delta + 1$ is an upper bound for the dichromatic number of any
$\Delta$-regular digraph.

The following was notably conjectured by Erd\H{o}s~\cite{Erd79b}.

\begin{conjecture} [Erd\H{o}s] \label{conjecture:erdos}
For every digon-free digraph $D$ with maximum total degree $\Delta$, we have $\vec{\chi}(D) = O(\frac{\Delta}{\log_2 \Delta})$.
\end{conjecture}

In this section, we study improvements of Brooks' Theorem for digon-free digraphs using the following substitute for the maximum degree. 
Given a digraph $D$, we denote 
\begin{equation*}
\tDelta = \tDelta(D) = \max\left\{\sqrt{d^+(v)d^-(v)}:v\in V(D)\right\}.
\end{equation*}

In~\cite{HM11}, Harutyunyan and Mohar provided the following upper bound on the dichromatic number.

\begin{theorem}[Harutyunyan and Mohar] \label{theorem:upbound-nonlist}
There is an absolute constant $\Delta_1$ such that every digon-free digraph $D$ with $\tDelta = \tDelta(D) \geq \Delta_1$ has $\vec{\chi}(D) \leq \left(1- e^{-13} \right) \tDelta$.
\end{theorem}

As the main result of this section, we extend Theorem~\ref{theorem:upbound-nonlist} to the list dichromatic number.
Namely, our result reads as follows.

\begin{theorem} \label{theorem:upbound-list}
There is an absolute constant $\Delta_1$ such that every digon-free digraph $D$ with $\tDelta = \tDelta(D) \geq \Delta_1$ has $\vec{\chi_\ell}(D) \leq \left(1- e^{-18} \right) \tDelta$.
\end{theorem}

\noindent The rest of this section is dedicated to the proof of Theorem \ref{theorem:upbound-list}. The proof is a modification of the argument in \cite{HM11}, which is itself based on an argument due to Molloy and Reed in~\cite{MR02} concerning the chromatic number of graphs. We first prove the following simple lemma.

\begin{lemma} \label{L:1:5}
Let $D$ be a digraph with maximum out-degree $\Delta_o$, and $L$ be a $(\Delta_o + 1 - r)$-list-assignment to $D$. Suppose we have a partial $L$-coloring of $D$ such that, for every vertex $v$, at least one of the followings holds:
\begin{enumerate}
	\item There are at least $r$ colors not in $L(v)$ that appear in $N^+(v)$.
	\item There are at least $r$ colors in $L(v)$ that appear at least twice in $N^+(v)$.
\end{enumerate}
Then $D$ is $L$-colorable.
\end{lemma}

\begin{proof}
In both cases, the number of non-used colors (\textit{i.e.} not used on $N^+(v)$) that remain in $L(v)$ is greater than the number of uncolored out-neighbors of $v$.
The partial $L$-coloring can thus be extended greedily.
\end{proof}

\begin{proof}[Proof of Theorem \ref{theorem:upbound-list}]
We may assume that $c_1 \tDelta < d^{+}(v) < c_2\tDelta$ and $c_1
\tDelta < d^{-}(v) < c_2\tDelta$ for every $v \in V(D)$, where $c_1
= 1 - \frac{1}{3}e^{-16}$ and $c_2 = 1 + \frac{1}{3}e^{-16}$. If
not, we remove all the vertices $v$ not satisfying the above
inequality and obtain a coloring for the remaining graph with
$\left (1-e^{-18} \right) \tDelta$ colors. Now, if a vertex does
not satisfy the above condition either one of $d^{+}(v)$ or
$d^{-}(v)$ is at most $c_1 \tDelta$ or one of $d^{+}(v)$ or
$d^{-}(v)$ is at most $\frac{1}{c_2} \tDelta$. Note that $1 -
e^{-18}
> \max \{c_1, 1/c_2\}$. This ensures that there is a color in $L(v)$
that either does not appear in the in-neighborhood or does not
appear in the out-neighborhood of $v$, allowing us to complete the
coloring.

The core of the proof is probabilistic. We first consider that every vertex is given a list of size $L=\lfloor \tDelta/2 \rfloor$ colors. Our random process is as follows:

\begin{enumerate}
	\item Assign to each vertex a color from its list uniformly at random.
	\item Uncolor a vertex if it is on a monochromatic directed path of length at least $2$ (containing at least $3$ vertices).
	\item Uncolor every vertex that retains its color with probability $\frac{1}{2}$.
\end{enumerate}

Clearly, this results in a proper partial $L$-coloring of $D$ since we do not keep any monochromatic directed cycle in our digraph.
In our upcoming arguments, we classify the vertices into two types depending on the color list of their out-neighbors:

\begin{enumerate}
	\item Vertex $v$ is of type $1$ if $v$ has at least $\lfloor \frac{c_1\tDelta}{2} \rfloor$ out-neighbors which contain at least $\lfloor \frac{L}{2} \rfloor$ colors in their list that are \textit{not} in $L(v)$. 
	\item Vertex $v$ is of type $2$ if $v$ has at least $\lfloor \frac{c_1\tDelta}{2} \rfloor$ out-neighbors which contain at least $\lfloor \frac{L}{2} \rfloor$ colors in their list that are in $L(v)$.
\end{enumerate}

It is clear that every vertex in $D$ must belong to at least one of these two types (note that it can be of both type). In fact, this classification of vertices corresponds to the two cases listed in Lemma \ref{L:1:5}. Intuitively, if $v$ is of type $1$, meaning that $v$ has many neighbors whose lists are very different from $L(v)$, then it is likely that some colors not in $L(v)$ will eventually appear in $N^+(v)$. In this case, we are interested in the random variable $Y_v$ defined as the number of colors \textit{not} in $L(v)$ which are assigned to at least one out-neighbor of $v$ and are retained by all of these vertices. On the other hand, if $v$ is of type $2$, we expect some colors of $L(v)$ to be repeated in $N^+(v)$ since the lists of the out-neighbors of $v$ are very similar to $L(v)$. 
Let $X_v$ be the number of colors in $L(v)$ which are assigned to at least two out-neighbors of $v$ and are retained by all of these vertices. 
Let $A_v$ be the event that both $Y_v$ and $X_v$ are less than $\frac{1}{2}e^{-16}\tDelta + 1$.
Having an event $A_v$ occurring is "bad" in our context; but we will prove that there is a chance that none
of the $A_v$'s occurs using the following version of the Lov\'{a}sz Local Lemma.

\begin{theorem}[Lov\'{a}sz Local Lemma -- Symmetric version] \label{prop: Lovasz}
Let $A_1,..., A_n$ be a finite set of events in some probability space $\Omega$ such that each $A_i$ occurs with probability at most $p$, and each $A_i$ is mutually independent of all but at most $d$ other events $A_j$. If $4pd \leq 1$, then $\PP \left[\cap_{i=1}^{n} \overline{A_i} \right] > 0$.
\end{theorem}

\noindent More precisely, we will show, using the Lov\'{a}sz Local Lemma, that with positive probability none of the events $A_v$ occurs, meaning that for every vertex $v$, either $Y_v$ or $X_v$ has to be bigger than $\frac{1}{2}e^{-16}\tDelta + 1$. Then Lemma \ref{L:1:5} will imply that $$\vec{\chi_\ell}(D) \leq (c_2-\frac{1}{2}e^{-16}) \tDelta \leq  (1 - e^{-18})\tDelta,$$ finishing the proof. 

Note that the color initially assigned to a vertex $u$ can affect $Y_v$ (and $X_v$) only if $u$ and $v$ are joined by a path of length at most 3. Thus, $A_v$ is mutually independent of all except at most $$(2c_2\tDelta) + (2c_2\tDelta)^{2} + (2c_2\tDelta)^3 + (2c_2\tDelta)^4 + (2c_2\tDelta)^5 + (2c_2\tDelta)^6 \leq 150 \tDelta^{6}$$ other
events $A_w$. Therefore, by the symmetric version of the Local Lemma, it suffices to show that for each event $A_v$, we have $$4 \cdot 150 \tDelta^6 \PP[A_v] < 1.$$ We will show that $\PP[A_v] <\tDelta^{-7}$ by considering each type of vertices separately.

We first prove the following two lemmas concerning type-1 vertices.

\begin{lemma}\label{L:1:5:2}
For every type-1 vertex $v$, we have $\EE[Y_v] \geq e^{-16}\tDelta$.
\end{lemma}  

\begin{proof}
Let $Y'_v$ be the random variable denoting the number of colors not in $L(v)$ that are assigned to exactly one out-neighbor of $v$ and are retained by this vertex. As we clearly have $Y_v \geq Y'_v$, it suffices to consider $\EE[Y'_v]$. 

Note that any color $c\notin L(v)$ will be counted by $Y'_v$ if a vertex $u \in N^{+}(v)$ is colored $c$ and no other vertex in $S = N(u)\cup N^{+}(v)$ is assigned color $c$. This will give us a lower bound on $\EE[Y'_v]$. For $c\notin L(v)$, let $N_c$ be the number of out-neighbors of $v$ containing color $c$ in their list. Since $v$ is of type~$1$, we have $\sum_{c\notin L(v)}N_c \geq \frac{c_1\tDelta L}{4}$. Clearly, there are $N_c$ choices for the vertex $u$. The probability that no vertex in $S$ different from $u$ gets colored $c$ is at least $(1- \frac{1}{C})^{|S|} \geq  (1- \frac{1}{C})^{3c_2\tDelta}$. Note that $u$ also needs to avoid being uncolored during Step~$3$ of our random process so that its color is preserved. Therefore
\begin{eqnarray*}
\PP[\textrm{$c$ is counted in $Y'_v$}] &\geq& N_c \cdot \frac{1}{L} \cdot 
\left(1- \frac{1}{L}\right)^{3c_2 \tDelta} \left(\frac{1}{2}\right).\\
\end{eqnarray*}
Hence, by linearity of expectation,
\begin{eqnarray*}
\EE[Y'_v] &\geq& \left(\sum_{c\notin L(v)}N_c\right) \frac{1}{L} \cdot 
\left(1- \frac{1}{L}\right)^{3c_2 \tDelta} \left(\frac{1}{2}\right)\\
&\geq& \frac{c_1\tDelta L}{4} \frac{1}{L} e^{-7} \frac{1}{2} \\
&\geq& e^{-16}\tDelta
\end{eqnarray*}
for $\tDelta$ sufficiently large.
\end{proof}

\begin{lemma}\label{L:1:5:3}
For every type-1 vertex $v$, we have $$\PP \left[ | Y_v - \EE[Y_v] | > \log_2 \tDelta \sqrt{\EE[Y_v]} \,
\right] < \tDelta^{-7}.$$
\end{lemma}

\begin{proof}
Let $AY_v$ be the random variable counting the number of colors not in $L(v)$
assigned to at least one out-neighbor of $v$, and $DY_v$ be the
random variable that counts the number of colors not in $L(v)$ assigned to at
least one out-neighbor of $v$ but removed from at least one of
them. Clearly, $Y_v = AY_v - DY_v$ and therefore it suffices to
show that both $AY_v$ and $DY_v$ are sufficiently concentrated
around their mean. We will show that for $t = \frac{1}{2} (\log_2
\tDelta) \sqrt{\EE[Y_v]}$, the following estimates hold:

\medskip

\noindent\textbf{Claim 1.1.} $\PP \left[|AY_v - \EE[AY_v]| > t \right] < 2 e^{-t^2/(8
\tDelta)}$.

\medskip

\noindent\textbf{Claim 1.2.} $\PP \left[|DY_v - \EE[DY_v]| > t \right]
  < 4 e^{-t^2/(200 \tDelta)}$.

\medskip
\noindent The two inequalities in Claims~1.1 and~1.2 yield that, for $\tDelta$
sufficiently large,
\begin{eqnarray*}
\PP[ | Y_v - \EE[Y_v] | > \log_2 \tDelta \sqrt{\EE[Y_v]}] &\leq& 2
e^{-\frac{t^2}{8 \tDelta}} + 4 e^{-\frac{t^2}{200 \tDelta}}\\
&\leq& \tDelta^{- \log_2 \tDelta}\\
&<& \tDelta^{-7},
\end{eqnarray*}
as we require. So it remains to establish both claims.

To prove Claim 1.1, we use a version of Azuma's Inequality from
\cite{MR02} called the Simple Concentration Bound. 

\begin{theorem}[Simple Concentration Bound] \label{th:1.5}
Let $X$ be a random variable determined by $n$ independent trials
$T_1,..., T_n$, and satisfying the property that changing the
outcome of any trial $T_i$ can affect $X$ by at most $c$. Then
$$\PP \left[|X-\EE[X]| > t \right] \leq 2e^{-\frac{t^2}{2c^2n}}. $$
\end{theorem}

Note that $AY_v$ depends only on the colors assigned to the
out-neighbors of $v$. Furthermore, each random choice can affect
$AY_v$ by at most 1. Therefore, we can take $c=1$ in the Simple
Concentration Bound for $X=AY_v$. Since the random assignments
of colors are made independently over the vertices, and we have
$d^{+}(v) \leq c_2 \tDelta$, Claim~1.1. is immediately deduced.

For Claim 1.2, we use the following variant of Talagrand's
Inequality. 

\begin{theorem}[Talagrand's Inequality] \label{th:2}
Let $X$ be a nonnegative random variable, not equal to 0, which is
determined by $n$ independent trials $T_1,\dots,T_n$, and
satisfies the following conditions for some $c,r > 0$:
\begin{enumerate}
\item Changing the outcome of any trial $T_i$ can affect $X$ by
at most $c$. 
\item If $X \geq s$ for any $s$, then there are at most
$rs$ trials whose exposure certifies that $X \geq s$.
\end{enumerate}
Then, for any $0 \leq \lambda \leq \EE[X]$, we have
$$ \PP \left[|X-\EE[X]| >  \lambda + 60c \sqrt{r\EE[X]} \, \right]
   \leq 4e^{-\frac{\lambda^2}{8c^2r\EE[X]}}.
$$
\end{theorem}

We apply Talagrand's Inequality to the random variable $DY_v$.
Note that we can take $c=1$ since any single random color
assignment can affect $DY_v$ by at most 1. Now suppose that
$DY_v \geq s$. One can certify that $DY_v \geq s$ by exposing,
for each of the $s$ colors $i$, one random color assignments in
$N^{+}(v)$ that certify that at least one vertex $u$ got color $i$. If $u$ lost its color during Step 2 of our random process, then we can expose two other color assignments to show that, 
or if $u$ lost its color during Step~3, we can expose that random choice to prove that $u$ is uncolored. Therefore, $DY_v
\geq s$ can be certified by exposing $4s$ random choices, and
hence we may take $r=4$ in Talagrand's Inequality. Note that $$t=
\frac{1}{2} \log_2 \tDelta \sqrt{\EE[Y_v]}
>\!\!> 60c \sqrt{r\EE[DY_v]} $$ since $\EE[Y_v] \geq \tDelta/e^{16}$ and
$\EE[DY_v] \leq c_2 \tDelta$. Now, taking $\lambda$ in
Talagrand's Inequality to be $\lambda = \frac{1}{2}t$, we obtain
that $$\PP[|DY_v -\EE[DY_v]| > t] \leq \PP[|DY_v-\EE[DY_v]| >
\lambda + 60c \sqrt{r\EE[DY_v]}].$$ Therefore, provided that $ \lambda
\leq \EE[DY_v]$, Claim~1.2 is confirmed.

It is sufficient to show that $\EE[DY_v] = \Omega (\tDelta)$
since $\lambda = O(\log_2 \tDelta \sqrt{\tDelta})$. It is evident that a color $c$ will be counted in $DY_v$ if it is assigned to \textit{exactly} one out-neighbor $u$ of $v$ and also removed from this vertex. Note that if the color list of $u$ is very different from the lists of its neighbors, it is possible that $u$ cannot get uncolored during Step~2 of our random process. However, vertex $u$ always has a positive probability to get uncolored because of Step~3. So we have
\begin{eqnarray*}
\PP[\textrm{$c$ is counted in $DY_v$}] &\geq& N_c \cdot \frac{1}{L} \cdot
\left(1- \frac{1}{L}\right)^{3c_2 \tDelta} \left(\frac{1}{2}\right).
\end{eqnarray*}
Hence, by linearity of expectation,
\begin{eqnarray*}
\EE[DY_v] &\geq& \left(\sum_{c\notin L(v)}N_c\right) \cdot \frac{1}{L} 
\left(1- \frac{1}{L}\right)^{3c_2 \tDelta} \left(\frac{1}{2}\right)\\
&\geq& \frac{c_1\tDelta L}{4} \frac{1}{L} e^{-7} \frac{1}{2} \\
&\geq& e^{-16}\tDelta
\end{eqnarray*}
for $\tDelta$ sufficiently large. Therefore, $\EE[DX_v] = \Omega (\tDelta)$ as required.
\end{proof}
\medskip

Since $\EE[Y_v] \leq c_2 \tDelta$, Lemmas \ref{L:1:5:2} and
\ref{L:1:5:3} imply $$\PP[A_v] < \PP[Y_v<\frac{1}{2}e^{-16}\tDelta+1]< \tDelta^{-7}$$ for every type-1 vertex $v$. 
We will now prove the similar claim for all type-2 vertices by proving the following two lemmas.

\begin{lemma} \label{L:1:5:4}
For every type-2 vertex $v$, we have $\EE[X_v] \geq e^{-16}\tDelta$.
\end{lemma}

\begin{proof}
Let $X'_v$ be the random variable denoting the number of colors in $L(v)$
that are assigned to exactly two out-neighbors of $v$ and are
retained by both of these vertices. Clearly $X_v \geq X'_v$ and
therefore it suffices to consider $\EE[X'_v]$.

Note that any color $c\in L(v)$ will be counted by $X'_v$ if two vertices $u,w
\in N^{+}(v)$ are colored $c$ and no other vertex in $S = N(u)
\cup N^{+}(v) \cup N(w)$ is assigned color $c$. From this, we will deduce
a lower bound on $\EE[X'_v]$. Clearly, there are $\binom{N_c}{2}$ choices for the set
$\{u,w\}$. Since $v$ is of type 2, we have $\sum_{c\in L(v)}N_c \geq \frac{c_1\tDelta L}{4}$. The probability that no vertex in $S$ gets colored $i$ is at least $(1- \frac{1}{C})^{|S|} \geq  (1- \frac{1}{C})^{5c_2
\tDelta}$. Note that these two vertices also have to keep their color after Step 3. Therefore:
\begin{eqnarray*}
\PP[\textrm{c is counted in $X'_v$}] &\geq& \binom{N_c}{2} \left( \frac{1}{L}
\right)^2 \left(1- \frac{1}{L}\right)^{5c_2 \tDelta} \left(\frac{1}{2}\right)^2.
\end{eqnarray*}
By linearity of expectation and the fact that $$L\left(\sum_{c\in L(v)}{N_c}^2\right) \geq \left(\sum_{c\in L(v)}N_c\right)^2 \geq \frac{{c_1}^2\tDelta^2L^2}{16},$$ we get
\begin{eqnarray*}
\EE[X'_v] &\geq& \left(\sum_{c\in L(v)}\frac{{N_c}^2}{3}\right) \left( \frac{1}{L}
\right)^2 \left(1- \frac{1}{L}\right)^{5c_2 \tDelta} \left(\frac{1}{2}\right)^2\\
&\geq& \frac{1}{3} \frac{{c_1}^2\tDelta^2L}{16} \left( \frac{1}{L} \right)^2 e^{-11} \frac{1}{4}\\
&\geq& e^{-16}\tDelta
\end{eqnarray*}
for $\tDelta$ sufficiently large.
\end{proof}

\begin{lemma} \label{L:1:5:5}
For every type-2 vertex $v$, we have $$\PP \left[ | X_v - \EE[X_v] | > \log_2 \tDelta \sqrt{\EE[X_v]} \,
\right] < \tDelta^{-7}.$$
\end{lemma}

\begin{proof}
Let $AX_v$ be the random variable counting the number of colors in $L(v)$
assigned to at least two out-neighbors of $v$, and $DX_v$ the
random variable that counts the number of colors in $L(v)$ assigned to at
least two out-neighbors of $v$ but removed from at least one of
them. Clearly $X_v = AX_v - DX_v$ and therefore it suffices to
show that $AX_v$ and $DX_v$ are sufficiently concentrated
around their mean. Similarly to the previous case, we will show that for $t = \frac{1}{2} (\log_2
\tDelta) \sqrt{\EE[X_v]}$ the following estimates hold:

\medskip

\noindent\textbf{Claim 1.3.} $\PP \left[|AX_v - \EE[AX_v]| > t \right] < 2 e^{-t^2/(8
\tDelta)}$.

\medskip

\noindent\textbf{Claim 1.4.} $\PP \left[|DX_v - \EE[DX_v]| > t \right]
  < 4 e^{-t^2/(200 \tDelta)}$.

\medskip
\noindent Again, for $\tDelta$ sufficiently large,
\begin{eqnarray*}
\PP[ | X_v - \EE[X_v] | > \log_2 \tDelta \sqrt{\EE[X_v]}] &\leq& 2
e^{-\frac{t^2}{8 \tDelta}} + 4 e^{-\frac{t^2}{200 \tDelta}}\\
&\leq& \tDelta^{- \log_2 \tDelta}\\
&<& \tDelta^{-7},
\end{eqnarray*}
as required. So it remains to establish both claims.

Claim 1.3 is proved exactly in the same way as Claim~1.1 was proved, namely using the Simple Concentration Bound. To prove Claim 1.4, we apply Talagrand's Inequality to the random variable $DX_v$. Note that we can take $c=1$ since any single random color
assignment can affect $DX_v$ by at most 1. Now suppose that
$DX_v \geq s$. One can certify that $DX_v \geq s$ by exposing,
for each of the $s$ colors $i$, two random color assignments in
$N^{+}(v)$ that certify that at least two vertices got color $i$. If one vertex lost its color during Step~2 of our random process, we can expose two other color assignments to certify, or if it gets uncolored during Step 3, we can expose the random choice in this step to prove that it is uncolored. Therefore, $DX_v
\geq s$ can be certified by exposing $5s$ random choices, and
hence we may take $r=5$ in Talagrand's Inequality. Note that $$t=
\frac{1}{2} \log_2 \tDelta \sqrt{\EE[X_v]}
>\!\!> 60c \sqrt{r\EE[DX_v]} $$ since $\EE[X_v] \geq \tDelta/e^{16}$ and
$\EE[DX_v] \leq c_2 \tDelta$. Now, taking $\lambda = \frac{1}{2}t$ in
Talagrand's Inequality, we obtain
that $$\PP[|DX_v -\EE[DX_v]| > t] \leq \PP[|DX_v-\EE[DX_v]| >
\lambda + 60c \sqrt{r\EE[DX_v]}].$$ Therefore, provided that $ \lambda
\leq \EE[DX_v]$, Claim~1.4 is confirmed.

It is now sufficient to show that $\EE[DX_v] = \Omega (\tDelta)$,
since $\lambda = O(\log_2 \tDelta \sqrt{\tDelta})$. It is evident that any color $c$ will be counted in $DX_v$ if it is assigned to \textit{exactly} two out-neighbors of $v$ and also removed from at least one of them. By the same argument as in the previous case, a vertex always has a positive probability to get uncolored during Step~3 of our random process. So:
\begin{eqnarray*}
\PP[\textrm{c is counted in $DX_v$}] &\geq& \binom{N_c}{2} \left( \frac{1}{L}
\right)^2 \left(1- \frac{1}{L}\right)^{5c_2 \tDelta} \left(\frac{1}{2}\right).
\end{eqnarray*}
By linearity of expectation, we deduce
\begin{eqnarray*}
\EE[DX_v] &\geq& \left(\sum_{c\in L(v)}\frac{{N_c}^2}{3}\right) \left( \frac{1}{L}
\right)^2 \left(1- \frac{1}{L}\right)^{5c_2 \tDelta} \left(\frac{1}{2}\right)\\
&\geq& \frac{1}{3} \frac{{c_1}^2\tDelta^2L}{16} \left( \frac{1}{L} \right)^2 e^{-11} \frac{1}{2}\\
&\geq& e^{-16}\tDelta
\end{eqnarray*}
for $\tDelta$ sufficiently large. Therefore, $\EE[DX_v] = \Omega (\tDelta)$ as required. 
\end{proof}
\medskip

Now, since $\EE[X_v] \leq c_2 \tDelta$, Lemmas~\ref{L:1:5:4} and~\ref{L:1:5:5} imply that $$\PP[A_v] < \PP \left[X_v<\frac{1}{2}e^{-16}\tDelta+1 \right]< \tDelta^{-7}$$ for every type-2 vertex $v$. 
This completes the proof of Theorem \ref{theorem:upbound-list}.
\end{proof}

\section{Conclusion} \label{section:ccl}

We conclude the paper with a question concerning oriented planar graphs.
A conjecture of Neumann-Lara \cite{NL85} states the following.

\begin{conjecture} [Neumann-Lara]
For every planar oriented graph $D$,
we have $\vec{\chi}(D) \leq 2$.
\end{conjecture}

It is easy to show by using degeneracy that for every planar oriented graph $D$,
we have $\vec{\chi}(D) \leq \vec{\chi_\ell}(D) \leq 3$. 
We wonder how large can the list dichromatic number of a planar oriented graph be.
In particular:

\begin{question}
Are there planar oriented graphs $D$ with $\vec{\chi_\ell}(D)= 3$?
\end{question}

\section*{Acknowledgements}

The authors are grateful to St\'{e}phan Thomass\'{e} for interesting discussions
as well as for the idea of the proof of Theorem \ref{theorem:lower-random-bipartite}.


\begin{thebibliography}{99}

\bibitem{AK98} N. Alon, and M. Krivelevich. The choice number of random bipartite graphs. \textit{Annals of Combinatorics}, 2(4):291-297, 1998.

\bibitem{AKS99} N. Alon, M. Krivelevich, and B. Sudakov. List coloring of random and pseudo-random graphs. \textit{Combinatorica}, 19(4):453-472, 1999.

\bibitem{AS04} N. Alon, and J. H. Spencer. The probabilistic method. \textit{John Wiley \& Sons}, 2004.

\bibitem{BFJKM04} D. Bokal, G. Fijav\v{z}, M. Juvan, P.M. Kayll, and B. Mohar. The circular chromatic number of a digraph. \textit{Journal of Graph Theory}, 46:227-240, 2004.

\bibitem{Erd79b} P. Erd\H{o}s. Problems and results in number theory and graphs theory. In \textit{Proceedings of the 9th Manitoba Conference on Numerical Mathematics and Computing}, pages 3–21, 1979.

\bibitem{EGK91}
P.~Erd\H{o}s, J. Gimbel, and D. Kratsch. Some extremal results in cochromatic
and dichromatic theory. \textit{Journal of Graph Theory}, 15:579-585, 1991.

\bibitem{EM64} P. Erd\H{o}s, and L. Moser. On the representation of directed 
graphs as unions of orderings. \textit{Magyar Tud. Akad. Mat. Kutato Int. Kozl.},
9:125--132, 1964.

\bibitem{ERT79} P. Erd\H{o}s, A.L. Rubin, and H. Taylor. Choosability in graphs. In \textit{Proceedings of the West Coast Conference on Combinatorics, Graph Theory and Computing, Congressus Numerantium 26}, pages 125–157, 1979.


\bibitem{H11} A.~Harutyunyan. Brooks-type results for coloring of digraphs. Ph.D.
thesis, Simon Fraser University, 2011.

\bibitem{HM11} A. Harutyunyan, and B. Mohar. Strengthened Brooks Theorem for digraphs of girth three. \textit{Electronic Journal of Combinatorics}, 18:\#P195, 2011.

\bibitem{HM11b} A. Harutyunyan, and B. Mohar. Gallai's Theorem for List Coloring of Digraphs. \textit{SIAM Journal on Discrete Mathematics}, 25(1):170-180, 2011.

\bibitem{HM12} A. Harutyunyan, and B. Mohar. Two results on the digraph
chromatic number. \textit{Discrete Mathematics}, 312(10):1823-1826, 2012. 

\bibitem{HM15} A. Harutyunyan, and B. Mohar. Planar digraphs of digirth five are 2-colorable. To appear in \textit{Journal of Graph Theory}.

\bibitem{Joh96} A. Johansson. Asymptotic choice number for triangle free graphs. DIMACS Technical Report 91-5, 1996.

\bibitem{O02} K. Ohba. On chromatic-choosable graphs. \textit{Journal of Graph Theory}, 40(2):130-135, 2002.

\bibitem{MT84} U. Manber, and M. Tompa. The effect of number of Hamiltonian paths on the complexity of a vertex-coloring problem. \textit{SIAM Journal on Computing}, 13(1):109–115, 1984.

 \bibitem{M2003} B. Mohar. Circular colorings of edge-weighted graphs. \textit{Journal of Graph Theory}, 43:107-116, 2003.

\bibitem{M10} B. Mohar. Eigenvalues and colorings of digraphs.
\textit{Linear Algebra and its Applications}, 432(9):2273-2277, 2010.

\bibitem{MR02} M. Molloy, and B. Reed. Graph colouring and the probabilistic method, volume 23. \textit{Springer Science \& Business Media}, 2002.

\bibitem{Neu82} V. Neumann-Lara. The dichromatic number of a digraph. \textit{Journal of Combinatorial Theory, Series B}, 33:265-270, 1982.

\bibitem{NL85} V. Neumann-Lara, Vertex colourings in digraphs. Some Problems.
Technical Report, University of Waterloo, July 8, 1985. 


\bibitem{NRW14} J. A. Noel, B. A. Reed, and H. Wu. A proof of a conjecture of Ohba. \textit{Journal of Graph Theory}, 79(2):86-102, 2014.

\bibitem{SS08} J. Spencer, and C. Subramanian. On the size of induced acyclic subgraphs in random digraphs. \textit{Discrete Mathematics and Theoretical Computer Science}, 10(2):47-54, 2008.

\end{thebibliography}
\end{document}